\documentclass{amsart}
\usepackage{amsmath}
\usepackage{amsfonts}
\usepackage{amssymb}
\usepackage[dvipdfmx]{graphicx}
\usepackage{mathrsfs}

\newtheorem{thm}{Theorem}[section]

\newtheorem{prop}[thm]{Proposition}

\newtheorem{lemma}[thm]{Lemma}

\newtheorem{cor}[thm]{Corollary}

\newtheorem{rem}{Remark}

\DeclareMathOperator{\ext}{Ext}

\DeclareMathOperator{\id}{id}

\DeclareMathOperator{\im}{Im}

\begin{document}

\title{The surface of circumferences for Jenkins-Strebel differentials}
\author{Masanori Amano}
\address{School of Management and Information University of Shizuoka 52-1 Yada, Suruga-ku, Shizuoka 422-8526, Japan}
\email{amano.m.ab@u-shizuoka-ken.ac.jp}

\begin{abstract}
There are some existence problems of Jenkins-Strebel differentials on a Riemann surface.
The one of them is to find a Jenkins-Strebel differential whose characteristic ring domains have given positive numbers as their circumferences, for any fixed underlying Riemann surface and core curves of the ring domains.
However, the solution may not exist for some given underlying surface, core curves, and positive numbers.
In this paper, we investigate the existence of such the solutions.
Our method is to use the surface of circumferences, which is determined by the extremal problem for Jenkins-Strebel differentials.
We can see degenerations of the characteristic ring domains of Jenkins-Strebel differentials by the surface.
Moreover, we also consider the behavior of the surface when the underlying Riemann surface varies.
\end{abstract}

\maketitle

\section{Introduction}

A holomorphic quadratic differential $\varphi $ on a Riemann surface $\Sigma $ admits a locally Euclidean structure on $\Sigma $.
The structure is obtained by a local coordinate $z$ of $\Sigma $ such that $\varphi =dz^2$.
Such $z$ also determines a (vertical) trajectory $z=\gamma (t)$ which is an analytic curve and satisfies $\varphi (\gamma (t))\{\gamma '(t)\}^2<0$.

There is a specific holomorphic quadratic differential $\varphi $, called a Jenkins-Strebel differential whose trajectories are either simple closed curves or curves joining critical points of $\varphi $.
By such $\varphi $, the surface $\Sigma $ is decomposed into a finite number of ring domains called the characteristic ring domains of $\varphi $, and the null set which consists of critical points and trajectories joining them.

A ring domain is conformally equivalent to a Euclidean cylinder, and it has three quantities; the circumference, the height, and the ratio of them, called the modulus.
There are three existence problems for Jenkins-Strebel differentials which use these quantities.
Let $\{\gamma _1,\ldots ,\gamma _k\}$ be a family of non-zero, non-peripheral, non-intersecting, and non-homotopic simple closed curves on $\Sigma $.
We call it an admissible curve family.
For any admissible curve family $\{\gamma _1,\ldots ,\gamma _k\}$ on $\Sigma $ and any positive numbers $x_1,\ldots ,x_k$, we consider the existence of a Jenkins-Strebel differential $\varphi $ on $\Sigma $ such that its characteristic ring domains admit $\{\gamma _1,\ldots ,\gamma _k\}$ as their core curves, and have $x_1,\ldots ,x_k$ as the pair of circumferences, heights, or (the ratio of) moduli of the ring domains with respect to the metric $|\varphi |^{\frac{1}{2}}$.
In the case of the heights, such $\varphi $ is always uniquely determined, and this result is called the height problem.
In the case of the moduli, such $\varphi $ also exists, it has the pair of moduli as $\lambda x_1,\ldots ,\lambda x_k$.
The constant $\lambda $ is uniquely determined, and $\varphi $ is also up to positive scalar multiples, this is called the moduli problem.
However, such $\varphi $ may not exist in the case of the circumferences.
Roughly speaking, this happens when some $x_j$ is relatively too larger or smaller than the others.

Instead of the above case of the circumferences, we consider the existence of a family of non-overlapping ring domains $\{R_1,\ldots ,R_k\}$ on $\Sigma $ whose core curves are $\{\gamma _1,\ldots ,\gamma _k\}$ (which allow the degeneration of each the ring domain, that is, $R_j=\varnothing $), such that it maximizes the sum $\sum A_j^2M_j$ where $M_j$ is the modulus of $R_j$, for any given positive numbers $A_1,\ldots ,A_k$.
Such family $\{R_j\}_{j=1}^k$ is always uniquely determined, and there is a Jenkins-Strebel differential $\varphi $ on $\Sigma $ whose norm coincides with the maximum sum, and whose characteristic ring domains coincide with $\{R_j\}_{j=1}^k$ such that each the circumference of $R_j$ is equal to $A_j$ if $R_j$ is non-degenerated.
However, such solution $\varphi $ may have some degenerated ring domains.
This result is called the extremal problem (Theorem \ref{Strebel_ext_prob}).

In 1950-80, Hubbard and Masur (\cite{HubMas76}), Jenkins (\cite{Jenkins57}, \cite{Jenkins58}), Jenkins and Suita (\cite{JenSui74}), Masur (\cite{Masur79}), Renelt (\cite{Renelt76}), and Strebel (\cite{Strebel66}, \cite{Strebel67}, \cite{Strebel76}) solved these the height, the moduli, and the extremal problems.
(Their works about quadratic differentials are summarized in the book of Strebel \cite{Strebel84}.)

In this paper, we investigate the behavior of degenerations of the characteristic ring domains in the extremal problem when the complex structure of the underlying Riemann surface $\Sigma $ varies.

Suppose that $\Gamma $ is an admissible curve family on $\Sigma $.
The space of all pairs of non-negative numbers $(A_1^2,\ldots ,A_k^2)$ such that any such point corresponds to a unit norm Jenkins-Strebel differential on $\Sigma $ whose core curves of the characteristic ring domains are $\Gamma $, is constructed by the extremal problem, and we denote the space by $\mathcal A(\Sigma ,\Gamma )$.
In \S \ref{surface_of_circumferences} of this paper, we define the space and consider its properties.
In the last of the chapter, we show that the space is a smooth convex hyper surface (Theorem \ref{surface_A}), then we call it the surface of circumferences for $\Gamma $ on $\Sigma $.

In next \S \ref{with_Teich}, the surface $\mathcal A(\Sigma ,\Gamma )$ is considered on the Teichm\"uller space.
Suppose that $p=(S,f)$ is a point in the Teichm\"uller space of $\Sigma $, where $S$ is a Riemann surface and $f:\Sigma \rightarrow S$ is a quasiconformal mapping, so that the surface $\mathcal A(S,f(\Gamma ))$ is similarly defined.
We show that $\mathcal A(S,f(\Gamma ))$ is an invariant on the Teichm\"uller space (Theorem \ref{invariant} and Corollary \ref{invariant_corollary}).

Finally, we consider about the surface $\mathcal A(S,f(\Gamma ))$ in the case of $\# \Gamma =2$ in \S \ref{dim2}.
In this situation, we can describe the form of the subsurface of $\mathcal A(S,f(\Gamma ))$ whose the corresponding Jenkins-Strebel differentials have all non-degenerated characteristic ring domains (Theorem \ref{form_of_A}).
Therefore, by all positive scalar multiples of the subsurface, it yields a domain on $\mathbf R_{\geq 0}^2$ whose each point coincides with (the squares of) circumferences of a (not necessarily unit norm) Jenkins-Strebel differential whose ring domains are all non-degenerated.
We also show that the domain tends to the empty set if $p$ tends to a specific point of the boundary of the Teichm\"uller space (Theorem \ref{sequence0}).

\section*{Acknowledgments}

The author is grateful to Hideki Miyachi for his useful and helpful comments and discussions.

\section{Jenkins-Strebel differentials and their properties}

\subsection{Jenkins-Strebel differentials}

Let $\Sigma =\Sigma _{g,n}$ be a Riemann surface of genus $g$ with $n$ punctures such that $3g-3+n>0$.
Let $\varphi $ be a tensor on $\Sigma $ of the form $\varphi =\varphi (z)dz^2$ where $\varphi (z)$ is a holomorphic function for a local coordinate $z$ of $\Sigma $.
We call $\varphi $ a \textit{holomorphic quadratic differential} on $\Sigma $.
This differential has the 2-form $|\varphi |=|\varphi (z)|dxdy$ where $z=x+iy$, and it has the $L^1$-norm $\|\varphi \|=\iint _{\Sigma }|\varphi |$.
The differential has finitely many zeros and poles.
The poles are only at punctures of $\Sigma $, and we call any zero or any pole of $\varphi $ a \textit{critical point} of $\varphi $.
The norm of $\varphi $ is finite if and only if the orders of poles of $\varphi $ are at most $1$.
We only deal with finite norm holomorphic quadratic differentials in the paper except for Theorem \ref{punctures} and in the proof of Proposition \ref{any_vector}.

For any non-zero holomorphic quadratic differential $\varphi $ on $\Sigma $, there exists a local coordinate $\zeta $ such that $\varphi $ is represented by $d\zeta ^2$.
Indeed, for any non-critical point $p_0\in \Sigma $ of $\varphi $, there is a pair $(U,z)$ of a simply connected neighborhood $U$ and a local coordinate $z$ at $p_0$ on $\Sigma $ such that $U$ contains no critical points of $\varphi $, and $\sqrt{\varphi (z)}$ is determined as a holomorphic single valued function on $U$.
Set $z_0=z(p_0)$ and construct
\begin{align*}
\zeta (z)=\int _{z_0}^z\sqrt{\varphi (z)}dz
\end{align*}
for any $z$, then it satisfies $d\zeta ^2=\varphi (z)dz^2$.
We call this $\zeta $ a \textit{$\varphi $-coordinate} or a \textit{natural parameter} of $\varphi $.
The family of all pairs of such neighborhoods and $\varphi $-coordinates $\{(U,\zeta )\}$ can determine an atlas which coincides with the complex structure on $\Sigma $.
The atlas has the locally Euclidean metric $|d\zeta |$ on the corresponding neighborhood, then we call it a \textit{flat structure} on $\Sigma $.

A maximal analytic curve $\gamma (t)$ on $\Sigma $ which satisfies $\varphi (\gamma (t))(d\gamma (t)/dt)^2<0$ is called a \textit{vertical trajectory} of $\varphi $.
We denote by $CT(\varphi )$ the set of all critical points of $\varphi $ and all vertical trajectories of $\varphi $ joining critical points, and call it a \textit{critical graph} of $\varphi $.
Any component of the surface $\Sigma -CT(\varphi )$ is either a ring domain which is surrounded by parallel closed vertical trajectories of $\varphi $, or a domain consists of infinitely many recurrent vertical trajectories of $\varphi $.
If $\Sigma -CT(\varphi )$ only has ring domains, we call $\varphi $ a \textit{Jenkins-Strebel differential}, and each of the ring domains a \textit{characteristic ring domain} of $\varphi $.
In particular, all closed vertical trajectories of a characteristic ring domain of $\varphi $ have a same length with respect to the metric $|\varphi |^{\frac{1}{2}}$.

For any ring domain $R$ on $\Sigma $, we call a simple closed curve on $\Sigma $ whose homotopy class generates a fundamental group of $R$ a \textit{core curve} of $R$.
For any Jenkins-Strebel differential $\varphi $, a characteristic ring domain of $\varphi $ with a core curve $\gamma $ is the unique maximal one among all ring domains whose core curves are homotopic to $\gamma $ such that the ring domains are swept out by closed vertical trajectories of $\varphi $.

Let $\gamma _1,\ldots ,\gamma _k$ be non-trivial, non-peripheral, non-intersect, and non-homotopic simple closed curves on $\Sigma $.
We call the set $\{\gamma _1,\ldots ,\gamma _k\}$ an \textit{admissible curve family} on $\Sigma $.
For example, if $\varphi $ is a Jenkins-Strebel differential on $\Sigma $, the set consists of each one core curve of each characteristic ring domain of $\varphi $ is an admissible curve family on $\Sigma $.
The maximal number of curves in any admissible curve family on $\Sigma $ is $3g-3+n$, and moreover, in the situation, the family is a set of pants curves of $\Sigma $.
Also the number of the characteristic ring domains of a Jenkins-Strebel differential on $\Sigma $ is at most $3g-3+n$.

\begin{rem}
The definition of a Jenkins-Strebel differential $\varphi $ on $\Sigma $ is equivalent to that the set of all non-closed trajectories and all critical points of $\varphi $ covers a null-set on $\Sigma $.
\end{rem}

\subsection{Some notations}

For any simple closed curve $\gamma $ and any ring domain $R$ on $\Sigma $, we call that $R$ is \textit{of type $\gamma $} if and only if a core curve of $R$ is homotopic to $\gamma $.
Let $\Gamma =\{\gamma _1,\ldots ,\gamma _k\}$ be any admissible curve family on $\Sigma $ and $\{R_j\}_{j=1}^k$ any family of ring domains on $\Sigma $ such that we allow some $R_j=\varnothing $ in the family.
We call that $\{R_j\}_{j=1}^k$ is \textit{of type $\Gamma $} if and only if any non-empty $R_j$ is type of $\gamma _j$.
If $R_j=\varnothing $ in the family, we call it a \textit{degenerated ring domain of type $\gamma _j$}.
In addition, we call that a Jenkins-Strebel differential $\varphi $ on $\Sigma $ is \textit{of type $\Gamma $} if and only if the family of the characteristic ring domains of $\varphi $ is of type $\Gamma $.
For such $\varphi $, we also allow that there are some degenerated ring domains.
By the above treatment of degenerated ring domains, if $\Gamma '\subset \Gamma $, we can say that a Jenkins-Strebel differential of type $\Gamma '$ is of type $\Gamma $.

We use some unified symbols for some quantities determined by Jenkins-Strebel differentials.
Let $\varphi $ be a Jenkins-Strebel differential on $\Sigma $ of type $\Gamma =\{\gamma _1,\ldots ,\gamma _k\}$, and $\{R_j\}_{j=1}^k$ the set of the characteristic ring domains (possibly including degenerated ones) of $\varphi $.
Suppose that $j=1,\ldots ,k$.
The symbol $M_j$ assigns to the modulus of $R_j$.
The symbol $a_j$ is
\begin{align*}
a_j=\inf _{\gamma '\sim \gamma _j}\int _{\gamma '}|\varphi |^{\frac{1}{2}},
\end{align*}
where the infimum ranges over all simple closed curves $\gamma '$ on $\Sigma $ homotopic to $\gamma _j$.
The symbol $b_j$ is $b_j=a_jM_j$.
If $R_j$ is non-degenerated, it is represented by a rectangle in the Euclidean plane by a $\varphi $-coordinate.
Any vertical line on the rectangle corresponds to a vertical trajectory of $\varphi $.
In this situation, the quantity $a_j$ means the vertical length of the rectangle, this is also the circumference of $R_j$ with respect to the metric $|\varphi |^{\frac{1}{2}}$.
Similarly, $b_j$ is the horizontal length, and is equal to the height of $R_j$.
The modulus $M_j$ is the ratio $b_j/a_j$.
Even if $R_j$ is degenerated, $a_j$ exists and is still positive because $\varphi \not =0$, but this is no longer the circumference of $R_j$, and the height and the modulus are $b_j=M_j=0$.

By the above notation, the norm of $\varphi $ holds
\begin{align*}
\|\varphi \|=\sum _{j=1}^ka_jb_j=\sum _{j=1}^ka_j^2M_j,
\end{align*}
since $\varphi $-coordinates implies that the norm of $\varphi $ is equal to all the Euclidean areas of $R_1,\ldots ,R_k$.

In this paper, we use these symbols $M$, $a$, and $b$ as the modulus, (possibly) the circumference, and the height of each characteristic ring domain of any Jenkins-Strebel differential, respectively, with sub scripts, or some additional marks.

\subsection{Existence problem for Jenkins-Strebel differentials}

The following two theorems are the most important tools for this paper.

\begin{thm}{(The extremal property \cite[Theorem 20.4]{Strebel84})}\label{Strebel_ext_prop}
Let $\Gamma =\{\gamma _1,\ldots ,\gamma _k\}$ be any admissible curve family on $\Sigma $, and $\varphi $ a Jenkins-Strebel differential on $\Sigma $ of type $\Gamma $.
Let $\{R_j\}_{j=1}^k$ be the family of the characteristic ring domains of type $\Gamma $ of $\varphi $, $a_j=\inf _{\gamma '\sim \gamma _j}\int _{\gamma '}|\varphi |^{\frac{1}{2}}$, and $M_j$ the modulus of $R_j$, for any $j=1,\ldots ,k$.
Suppose that $\{\tilde R_j\}_{j=1}^k$ is any family of (possibly degenerated) non-overlapping ring domains on $\Sigma $ of type $\Gamma $, and denote by $\tilde M_j$ the modulus of $\tilde R_j$ for any $j=1,\ldots ,k$.
Then the inequality 
\begin{align*}
\sum _{j=1}^ka_j^2\tilde M_j\leq \sum _{j=1}^ka_j^2M_j(=\|\varphi \|)
\end{align*}
holds, where the equality holds if and only if $\tilde R_j=R_j$ for any $j=1,\ldots ,k$.
\end{thm}

\begin{thm}{(The extremal problem, \cite[Theorem 21.10]{Strebel84})}\label{Strebel_ext_prob}
Let $\Gamma =\{\gamma _1,\ldots ,\gamma _k\}$ be any admissible curve family on $\Sigma $ and $A_1,\ldots ,A_k$ any positive numbers.
For any family of (possibly degenerated) non-overlapping ring domains $\{\tilde R_j\}_{j=1}^k$ on $\Sigma $ of type $\Gamma $, let $\tilde M_j$ be the modulus of $\tilde R_j$ for any $j=1,\ldots ,k$ and consider the sum $\sum _{j=1}^kA_j^2\tilde M_j$.
Then, there exists ring domains $\{R_j\}_{j=1}^k$ which maximize the sum, and they are the characteristic ones of a Jenkins-Strebel differential $\varphi $ on $\Sigma $ of type $\Gamma $.
Such $\varphi $ is uniquely determined.
Moreover, for the modulus $M_j$ of $R_j$, if $M_j>0$ then $A_j=a_j$ $(=\inf _{\gamma '\sim \gamma _j}\int _{\gamma '}|\varphi |^{\frac{1}{2}})$, and if $M_j=0$ then $A_j\leq a_j$.
\end{thm}

In the statement of Theorem \ref{Strebel_ext_prob}, the norm of $\varphi $ is also represented by
\begin{align*}
\|\varphi \|=\sum _{j=1}^ka_j^2M_j=\sum _{j=1}^kA_j^2M_j.
\end{align*}

\section{The surface of circumferences}\label{surface_of_circumferences}

Let $\Gamma $ be an admissible curve family on $\Sigma $.
We denote by $JS(\Gamma )$ the set of all Jenkins-Strebel differentials on $\Sigma $ of type $\Gamma $ with the $0$-quadratic differential.
Let $\mathbf R_{\geq 0}^k$ be the set of all pairs of $k$ non-negative real numbers, and $\mathbf o$ the origin $(0,\ldots ,0)$ in $\mathbf R_{\geq 0}^k$.

\subsection{The mapping $F$}

Let us fix an admissible curve family $\Gamma =\{\gamma _1,\ldots ,\gamma _k\}$ on $\Sigma $.
For any $(A_1,\ldots ,A_k)\in \mathbf R_{\geq 0}^k$, there exists a unique Jenkins-Strebel differential $\varphi \in JS(\Gamma )$ which satisfies the condition of Theorem \ref{Strebel_ext_prob}.
If some entries of $(A_1,\ldots ,A_k)$ are $0$, we omit the corresponding curves in $\Gamma $ and apply Theorem \ref{Strebel_ext_prob}.
Then the mapping $F:\mathbf R_{\geq 0}^k\rightarrow JS(\Gamma )$ is defined by the above correspondence with the special case $F(\mathbf o)=0$.

\begin{rem}\label{0}
We have that $(A_1,\ldots ,A_k)=\mathbf{o}$ if and only if $F(A_1,\ldots ,A_k)=0$.
If $F(A_1,\ldots ,A_k)=0$, every $a_j=\inf _{\gamma '\sim \gamma _j}\int _{\gamma '}|0|$ is immediately $0$, and by Theorem \ref{Strebel_ext_prob}, $A_j\leq a_j$, then $(A_1,\ldots ,A_k)=\mathbf{o}$.
The converse is by definition.
\end{rem}

\begin{rem}
Let $\varphi =F(A_1,\ldots ,A_k)$ and $\lambda >0$, then $\lambda \varphi =F(\lambda ^{\frac{1}{2}}A_1,\ldots ,\lambda ^{\frac{1}{2}}A_k)$.
This is also obtained by Theorem \ref{Strebel_ext_prob}.
Let $\{\tilde R_j\}_{j=1}^k$ be any family of non-overlapping ring domains of type $\Gamma $, and $\tilde M_j$ the modulus of $\tilde R_j$ for any $j=1,\ldots ,k$.
Then the maximum of the sum $\sum _{j=1}^k(\lambda ^{\frac{1}{2}}A_j)^2\tilde M_j$ is $\|\lambda \varphi \|$ since the maximum of $\sum _{j=1}^kA_j^2\tilde M_j$ is $\|\varphi \|$.
The uniqueness of $\lambda \varphi $ yields the desired result.
\end{rem}

We now check some properties for $F$.
As the topology of the set of all holomorphic quadratic differentials on $\Sigma $, we induce the locally uniformly convergence topology.
The topology of $JS(\Gamma )$ is also determined by it.

\begin{thm}\label{F_is_continuous}
The mapping $F$ is continuous.
\end{thm}


This theorem is implied by the following lemmas.
The readers can see \cite{Strebel84} for more details.
Let $JS_0(\Gamma )$ be the set of all elements in $JS(\Gamma )$ with norm $1$.

\begin{lemma}{(\cite[Theorem 21.2]{Strebel84})}\label{compact}
The space $JS_0(\Gamma )$ is compact.
\end{lemma}

\begin{lemma}{(\cite[Corollary 21.2]{Strebel84})}\label{converges}
Suppose that a sequence $\{\varphi _n\}_{n\in \mathbf N}$ in $JS(\Gamma )$ converges locally uniformly to $\varphi $.
Then $\varphi \in JS(\Gamma )$, and furthermore, the modulus $M_{jn}$, the quantity $a_{jn}=\inf _{\gamma '\sim \gamma _j}\int _{\gamma '}|\varphi _n|^{\frac{1}{2}}$, and the height $b_{jn}=a_{jn}M_{jn}$ determined by $\varphi _n$ also converge to the corresponding $M_j$, $a_j$, and $b_j$ by $\varphi $, respectively, for any $j=1,\ldots ,k$.
The sequence $\{\varphi _n\}_{n\in \mathbf N}$ also converges in norm.
\end{lemma}

\begin{proof}[Proof of Theorem \ref{F_is_continuous}]
Let $\{(A_{1n},\ldots ,A_{kn})\}_{n\in \mathbf N}$ be a sequence in $\mathbf R_{\geq 0}^k$ and suppose it converges to $(A_1,\ldots ,A_k)\in \mathbf R_{\geq 0}^k$ as $n\rightarrow \infty $.
We set $F(A_{1n},\ldots ,A_{kn})=\varphi _n$ and $F(A_1,\ldots ,A_k)=\varphi $.

We first consider as $(A_1,\ldots ,A_k)\not =\mathbf{o}$, hence the sequence $\{(A_{1n},\ldots ,A_{kn})\}_{n\in \mathbf N}$ can be non-zero for any $n$, and then $\varphi _n,\varphi \not =0$ by Remark \ref{0}.
Since $\varphi _n/\|\varphi _n\|$ is in $JS_0(\Gamma )$ and by Lemma \ref{compact}, it converges locally uniformly to some $\tilde \varphi \in JS_0(\Gamma )$ after passing to a subsequence if necessary.
Let $M_{jn}$ be the modulus of the characteristic ring domain of type $\gamma _j$ of $\varphi _n$, and set $a_{jn}=\inf _{\gamma '\sim \gamma _j}\int _{\gamma '}|\varphi _n|^{\frac{1}{2}}$.
Similarly, set $\tilde M_j,\tilde a_j$ and $M_j,a_j$ by $\tilde \varphi $ and $\varphi $, respectively.
Therefore, since $\varphi _n/\|\varphi _n\|$ converges locally uniformly to $\tilde \varphi $, we have
\begin{align*}
M_{jn}\rightarrow \tilde M_j, \frac{a_{jn}}{\|\varphi _n\|^{\frac{1}{2}}}\rightarrow \tilde a_j
\end{align*}
by Lemma \ref{converges}.
Let $\lambda =\sum _{j=1}^kA_j^2\tilde M_j$ then the limits
\begin{align*}
&\|\varphi _n\|=\sum _{j=1}^ka_{jn}^2M_{jn}=\sum _{j=1}^kA_{jn}^2M_{jn}\rightarrow \sum _{j=1}^kA_j^2\tilde M_j=\lambda ,\\
&a_{jn}=\|\varphi _n\|^{\frac{1}{2}}\cdot \frac{a_{jn}}{\|\varphi _n\|^{\frac{1}{2}}}\rightarrow \lambda ^{\frac{1}{2}}\tilde a_j
\end{align*}
hold.
We now use Theorem \ref{Strebel_ext_prop} and $A_{jn}\leq a_{jn}, A_j\leq a_j$ then
\begin{align*}
\|\varphi \|&=\sum _{j=1}^ka_j^2M_j\geq \sum _{j=1}^ka_j^2\tilde M_j\geq \sum _{j=1}^kA_j^2\tilde M_j=\lambda =\lambda \|\tilde \varphi \|\\
&=\lambda \sum _{j=1}^k\tilde a_j^2\tilde M_j\geq \lambda \sum _{j=1}^k\tilde a_j^2M_j=\lim _{n\rightarrow \infty }\sum _{j=1}^ka_{jn}^2M_j\geq \lim _{n\rightarrow \infty }\sum _{j=1}^kA_{jn}^2M_j\\
&=\sum _{j=1}^kA_j^2M_j=\sum _{j=1}^ka_j^2M_j=\|\varphi \|.
\end{align*}
Therefore the equation $\sum _{j=1}^ka_j^2M_j=\sum _{j=1}^ka_j^2\tilde M_j$ holds, so each characteristic ring domains of $\varphi $ and $\tilde \varphi $ coincide by Theorem \ref{Strebel_ext_prop}.
This implies that $\varphi $ and $\tilde \varphi $ coincide up to a scalar multiple.
On the other hand, the equation $\|\varphi \|=\lambda \|\tilde \varphi \|$ also holds, then we have $\varphi =\lambda \tilde \varphi $.
Finally, $\varphi _n=\|\varphi _n\|\cdot (\varphi _n/\|\varphi _n\|)$ converges locally uniformly to $\lambda \tilde \varphi =\varphi $.
This convergence does not depend on any subsequence of $\varphi _n/\|\varphi _n\|$.

The remaining case is if $(A_1,\ldots ,A_k)=\mathbf o$, that is $\varphi =0$.
Suppose that $\varphi _n$ does not converge locally uniformly to $0$.
There are a compact set $K$ on $\Sigma $, a local coordinate $z$ on $K$, and $\varepsilon >0$ such that for any $N\in \mathbf N$, there exists $n(N)>N$ then $\sup_{z\in K}|\varphi _{n(N)}(z)|\geq \varepsilon $.
Therefore $\varphi _{n(N)}\not =0$ for any $N\in \mathbf N$.
By Lemma \ref{compact}, $\varphi _{n(N)}/\|\varphi _{n(N)}\|$ converges locally uniformly to $\tilde \varphi \in JS_0(\Gamma )$ as $N\rightarrow \infty $ after passing to a subsequence if necessary.
We set suitable quantities $M_{jn(N)},a_{jn(N)}$ by $\varphi _{n(N)}$.
Since $M_{jn(N)}$ converges to some value by Lemma \ref{converges}, so that $\|\varphi _{n(N)}\|=\sum _{j=1}^ka_{jn(N)}^2M_{jn(N)}=\sum _{j=1}^kA_{jn(N)}^2M_{jn(N)}\rightarrow 0$.
Then the differential $\varphi _{n(N)}=\|\varphi _{n(N)}\|\cdot (\varphi _{n(N)}/\|\varphi _{n(N)}\|)$ converges locally uniformly to $0\cdot \tilde \varphi =0$.
This means that $\lim _{N\rightarrow \infty }\sup _{z\in K}|\varphi _{n(N)}(z)|=0$ but we already have that $\sup_{z\in K}|\varphi _{n(N)}(z)|\geq \varepsilon $, so this is a contradiction.
\end{proof}

\begin{rem}\label{converges_problem}
For \cite[Corollary 21.2]{Strebel84}, in fact, there is no direct description for that if $\varphi _n\in JS(\Gamma )$ converges locally uniformly to $\varphi \in JS(\Gamma )$, then $a_{jn}$ converges to $a_j$ when the $j$-th characteristic ring domain $R_j$ of $\varphi $ is degenerated.
However, this little bit problem is solved by using \cite[Theorem 24.7]{Strebel84}.
The theorem says that if a sequence of holomorphic quadratic differentials $\{\varphi _n\}_{n\in \mathbf N}$ on $\Sigma $ converges locally uniformly to $\varphi $, $\inf _{\gamma '\sim \gamma }\int _{\gamma '}|\im \{\varphi _n(z)^\frac{1}{2}dz\}|$ tends to $\inf _{\gamma '\sim \gamma }\int _{\gamma '}|\im \{\varphi (z)^\frac{1}{2}dz\}|$ for any closed curve $\gamma $ on $\Sigma $.

We go back to our situation.
For the quantity $a_j$, the equation
\begin{align*}
a_j=\inf _{\gamma '\sim \gamma _j}\int _{\gamma '}|\varphi (z)^{\frac{1}{2}}dz|=\inf _{\gamma '\sim \gamma _j}\int _{\gamma '}|\im \{\varphi (z)^{\frac{1}{2}}dz\}|
\end{align*}
holds.
Because $a_j$ is equal to either the length of a closed vertical trajectory of $\varphi $ (if $R_j$ is non-degenerated) or the length of the union of some vertical trajectories contained in $CT(\varphi )$ such that the union forms a closed curve homotopic to $\gamma _j$ (if $R_j$ is degenerated) with respect to the metric $|\varphi |^{\frac{1}{2}}$.
The readers confirm this fact in \cite{FatLauPoe79}.
This is of course true for any $a_{jn}$.
Therefore, by the theorem,
\begin{align*}
a_{jn}=\inf _{\gamma '\sim \gamma _j}\int _{\gamma '}|\im \{\varphi _n(z)^{\frac{1}{2}}dz\}|\rightarrow \inf _{\gamma '\sim \gamma _j}\int _{\gamma '}|\im \{\varphi (z)^{\frac{1}{2}}dz\}|=a_j
\end{align*}
even if the $j$-th characteristic ring domain of $\varphi $ is degenerated.

So our Lemma \ref{converges} contains the above result.
\end{rem}

We now consider the behavior of $F$ when only one entry of $(A_1,\ldots ,A_k)\in \mathbf R_{\geq 0}^k$ varies and the others are invariant.
Without loss of generality, we can choice the $k$-th entry $A_k$ as the variable, so the following two lemmas hold.

\begin{lemma}\label{lower}
Let $\varphi =F(A_1,\ldots ,A_{k-1},A_k)$ and suppose that the $k$-th characteristic ring domain of $\varphi $ is degenerated.
Then for any $0\leq A'_k\leq A_k$, the $k$-th characteristic ring domain of $F(A_1,\ldots ,A_{k-1},A'_k)$ is also degenerated.
Moreover, any $F(A_1,\ldots ,A_{k-1},A'_k)$ is equal to $\varphi $.
\end{lemma}

\begin{proof}
Let $\varphi '=F(A_1,\ldots ,A_{k-1},A'_k)$ and $M_j,a_j$ be the corresponding quantities by $\varphi $, also $M'_j,a'_j$ the ones by $\varphi '$.
We notice that $A_j\leq a_j,a'_j$ for any $j=1,\ldots ,k-1$, $A_k\leq a_k$, $A'_k\leq a'_k$, and $M_k=0$.
Then by Theorem \ref{Strebel_ext_prop}, we have
\begin{align*}
\|\varphi '\|&=\sum _{j=1}^k{a'_j}^2M'_j\geq \sum _{j=1}^k{a'_j}^2M_j\geq \sum _{j=1}^{k-1}A_j^2M_j\\
&=\sum _{j=1}^{k-1}A_j^2M_j+A_k^2M_k=\|\varphi \|=\sum _{j=1}^ka_j^2M_j\geq \sum _{j=1}^ka_j^2M'_j\\
&\geq \sum _{j=1}^{k-1}A_j^2M'_j+A_k^2M'_k\geq \sum _{j=1}^{k-1}A_j^2M'_j+{A'_k}^2M'_k=\|\varphi '\|.
\end{align*}
Therefore, the above inequalities are all equalities, then $\varphi '$ and $\varphi $ coincide up to a scalar multiple, but $\|\varphi '\|=\|\varphi \|$ makes the fact that $\varphi '=\varphi $.
\end{proof}

\begin{lemma}\label{upper}
Let $\varphi =F(A_1,\ldots ,A_{k-1},A_k)$ and suppose that the $k$-th characteristic ring domain of $\varphi $ is non-degenerated.
Then for any $A'_k\geq A_k$, the $k$-th characteristic ring domain of $F(A_1,\ldots ,A_{k-1},A'_k)$ is also non-degenerated.
\end{lemma}

\begin{proof}
The same notations for the previous lemma give $\varphi '=F(A_1,\ldots ,A_{k-1},A'_k)$, $A_j\leq a_j,a'_j$ for any $j=1,\ldots ,k-1$, $A_k\leq a_k$, $A'_k\leq a'_k$, and $M_k>0$.
By Theorem \ref{Strebel_ext_prop},
\begin{align*}
\|\varphi '\|&=\sum _{j=1}^k{a'_j}^2M'_j\geq \sum _{j=1}^k{a'_j}^2M_j\geq \sum _{j=1}^{k-1}A_j^2M_j+{A'_k}^2M_k\\
&=\sum _{j=1}^{k-1}A_j^2M_j+A_k^2M_k+({A'_k}^2-A_k^2)M_k=\|\varphi \|+({A'_k}^2-A_k^2)M_k\\
&=\sum _{j=1}^ka_j^2M_j+({A'_k}^2-A_k^2)M_k\geq \sum _{j=1}^ka_j^2M'_j+({A'_k}^2-A_k^2)M_k\\
&\geq \sum _{j=1}^kA_j^2M'_j+({A'_k}^2-A_k^2)M_k=\sum _{j=1}^{k-1}A_j^2M'_j+{A'_k}^2M'_k+({A'_k}^2-A_k^2)(M_k-M'_k)\\
&=\|\varphi '\|+({A'_k}^2-A_k^2)(M_k-M'_k).
\end{align*}
This implies that $({A'_k}^2-A_k^2)(M_k-M'_k)\leq 0$.
By the assumption $A'_k\geq A_k$, if $A'_k>A_k$ then $0<M_k\leq M'_k$, and if $A'_k=A_k$ then $\varphi '=\varphi $, so that $0<M_k=M'_k$.
Therefore we have $M'_k>0$ in any case.
\end{proof}

By combining the above two lemmas, we have the following proposition.

\begin{prop}
For any fixed $(A_1,\ldots ,A_{k-1})\in \mathbf R_{\geq 0}^{k-1}$, we have
\begin{align*}
&\sup\{A_k\mid \text{the $k$-th ring domain of $F(A_1,\ldots ,A_{k-1},A_k)$ is degenerated}\}\\
=&\inf\{A_k\mid \text{the $k$-th ring domain of $F(A_1,\ldots ,A_{k-1},A_k)$ is non-degenerated}\}.
\end{align*}
\end{prop}

Then we denote the value by $A^*_k=A^*_k(A_1,\ldots ,A_{k-1})$.
We investigate $A^*_k$ in the remaining section.

\begin{lemma}\label{lower+}
For any $0\leq A_k\leq A^*_k$, the differential $F(A_1,\ldots ,A_{k-1},A_k)$ does not depend on $A_k$.
\end{lemma}

\begin{proof}
If $A^*_k>0$, by Lemma \ref{lower}, every $F(A_1,\ldots ,A_{k-1},A_k)$ is same for any $0\leq A_k<A^*_k$.
By the continuity of $F$, $F(A_1,\ldots ,A_{k-1},A^*_k)$ is also same.
On the other hand, there is no consideration for the argument if $A^*_k=0$.
\end{proof}

\begin{prop}\label{characterize}
The value $A^*_k$ is characterized as
\begin{align*}
A^*_k=\inf _{\gamma '\sim \gamma _k}\int _{\gamma '}|F(A_1,\ldots ,A_{k-1},0)|^{\frac{1}{2}}.
\end{align*}
\end{prop}

\begin{proof}
We set
\begin{align*}
a_k(A_k)=\inf _{\gamma '\sim \gamma _k}\int _{\gamma '}|F(A_1,\ldots ,A_{k-1},A_k)|^{\frac{1}{2}}
\end{align*}
for all $A_k\geq 0$, then $A_k=a_k(A_k)$ if $A_k>A^*_k$ by Theorem \ref{Strebel_ext_prob} and Lemma \ref{upper}.
Let us vary $A_k$ continuously, then $F(A_1,\ldots ,A_{k-1},A_k)$ also varies continuously by the continuity of $F$, therefore $a_k(A_k)$ is continuous by Lemma \ref{converges}.
As $A_k\searrow A^*_k$, we have $A^*_k=a_k(A^*_k)$.

By Lemma \ref{lower+}, $F(A_1,\ldots ,A_{k-1},A^*_k)=F(A_1,\ldots ,A_{k-1},0)$.
Therefore, we have $A^*_k=a_k(A^*_k)=a_k(0)$.
\end{proof}

\begin{prop}\label{star_range}
For $A^*_k$, it holds $A^*_k\lneq \infty $ in any case, and holds $0\lneq A^*_k$ if it is neither $\# \Gamma =1$ nor $(A_1,\ldots ,A_{k-1})=\mathbf o$.
Therefore the mapping $F$ is not injective whenever $\# \Gamma \geq 2$.
\end{prop}

\begin{proof}
Proposition $\ref{characterize}$ says that if $A^*_k=0$ then $F(A_1,\ldots ,A_{k-1},0)=0$, so that this situation only happens in the case of either $\# \Gamma =1$ or $(A_1,\ldots ,A_{k-1})=\mathbf o$ by Remark \ref{0}.

If $A^*_k=\infty $, every $F(A_1,\ldots ,A_{k-1},A_k)$ is same for any $A_k\geq 0$, so we denote it by $\varphi $.
In particular, $F(A_1,\ldots ,A_{k-1},n)=\varphi $ is also holds for any $n\in \mathbf N$.
By multiplying the differential by $1/n^2$, so $F(A_1/n,\ldots ,A_{k-1}/n,1)=\varphi /n^2$ tends to $F(0,\ldots ,0,1)=0$ as $n\rightarrow \infty $ then this is a contradiction.
\end{proof}

\subsection{The surface of circumferences $\mathcal A$}

We define the space
\begin{align*}
\mathcal A=\{(A_1^2,\ldots ,A_k^2)\in \mathbf R_{\geq 0}^k\mid F(A_1,\ldots ,A_k)\in JS_0(\Gamma )\},
\end{align*}
and call it the \textit{surface of (the squares of) circumferences} for $\Gamma $ on $\Sigma $.

In Strebel's book \cite{Strebel84}, there are related surfaces called the surface of heights and the surface of moduli.
These are strictly concave and strictly convex, respectively.
The space $\mathcal A$ also has a similar property.
The proof of the following theorem is achieved by similar methods which are used for the surfaces of heights and moduli.

\begin{thm}\label{surface_A}
The space $\mathcal A$ is a smooth convex hyper surface.
\end{thm}

\begin{proof}
In order to prove that $\mathcal A$ is a hyper surface, it suffices to show that $\mathcal A$ is homeomorphic to the set of all non-negative points of the $k-1$ dimensional unit sphere 
\begin{align*}
\mathbf S_{\geq 0}^{k-1}=\{(e_1,\ldots ,e_k)\in \mathbf R^k\mid \sum _{j=1}^ke_j^2=1, e_j\geq 0 \text{ for any }j\}.
\end{align*}
Let $(e_1,\ldots ,e_k)\in \mathbf S_{\geq 0}^{k-1}$ and set $\varphi =F(e_1^{\frac{1}{2}},\ldots ,e_k^{\frac{1}{2}})\in JS(\Gamma )$.
By the normalization, $(e_1/\|\varphi \|,\ldots ,e_k/\|\varphi \|)\in \mathcal A$.
If $(e_1,\ldots ,e_k)$ varies continuously on $\mathbf S_{\geq 0}^{k-1}$, $\varphi $ is also on $JS(\Gamma )$ by the continuity of $F$.
Since $\varphi $ also converges in norm by Lemma \ref{converges}, the norm $\|\varphi \|$ also varies continuously.
Therefore $(e_1/\|\varphi \|,\ldots ,e_k/\|\varphi \|)$ does, and we obtain the continuous mapping $\mathbf S_{\geq 0}^{k-1}$ to $\mathcal A$ which maps $(e_1,\ldots ,e_k)$ to $(e_1/\|\varphi \|,\ldots ,e_k/\|\varphi \|)$.
The inverse is easily constructed as the mapping of $(e'_1,\ldots ,e'_k)\in \mathcal A$ to $(e'_1/\lambda ,\ldots ,e'_k/\lambda )\in \mathbf S_{\geq 0}^{k-1}$ where $\lambda =(\sum _{j=1}^k{e'_j}^2)^{\frac{1}{2}}$, and is obviously continuous, so the set $\mathcal A$ is a hyper surface on the Euclidean space $\mathbf R^k$.

The convexity and the smoothness of $\mathcal A$ is as follows.
Let $\mathbf A=(A_1^2,\dots ,A_k^2)$ and $\mathbf A_0=(A_{01}^2,\ldots ,A_{0k}^2)$ be any points in $\mathcal A$.
We now consider as $\mathbf A\rightarrow \mathbf A_0$.
Set $\varphi =F(A_1,\dots ,A_k)$ and $\varphi _0=F(A_{01},\ldots ,A_{0k})$, and let $\mathbf M=(M_1,\ldots ,M_k)$ and $\mathbf M_0=(M_{01},\ldots ,M_{0k})$ be vectors of moduli by $\varphi $ and $\varphi _0$, respectively.
We also denote by $a_j$ and $a_{0j}$ the quantities by $\varphi $ and $\varphi _0$, respectively.
By Theorem \ref{Strebel_ext_prop},
\begin{align*}
1=\|\varphi \|=\sum _{j=1}^kA_j^2M_j=\sum _{j=1}^ka_j^2M_j\geq \sum _{j=1}^ka_j^2M_{0j}\geq \sum _{j=1}^kA_j^2M_{0j}.
\end{align*}
Therefore, we can describe the above inequality as follows with the case of $\varphi _0$, by using $\cdot $ which is the inner product on the vector space $\mathbf R^k$:
\begin{align*}
1=\mathbf A\cdot \mathbf M\geq \mathbf A\cdot \mathbf M_0,\ 1=\mathbf A_0\cdot \mathbf M_0\geq \mathbf A_0\cdot \mathbf M.
\end{align*}
By combining the inequalities, it holds $\mathbf A\cdot \mathbf M_0\leq \mathbf A_0\cdot \mathbf M_0$ then
\begin{align}
(\mathbf A-\mathbf A_0)\cdot \mathbf M_0\leq 0.\label{convex_proof}
\end{align}
This means that the angle between the vectors $\mathbf M_0$ and $\mathbf A-\mathbf A_0$ is larger than or equal to $\pi /2$.
We also have $\mathbf A_0\cdot \mathbf M\leq \mathbf A\cdot \mathbf M$ and $-(\mathbf A-\mathbf A_0)\cdot \mathbf M\leq 0$, hence $(\mathbf A-\mathbf A_0)\cdot (\mathbf M_0-\mathbf M)\leq (\mathbf A-\mathbf A_0)\cdot \mathbf M_0\leq 0$.
Therefore
\begin{align*}
\frac{\mathbf A-\mathbf A_0}{\|\mathbf A-\mathbf A_0\|}\cdot (\mathbf M_0-\mathbf M)\leq \frac{\mathbf A-\mathbf A_0}{\|\mathbf A-\mathbf A_0\|}\cdot \mathbf M_0\leq 0.
\end{align*}
Let $\mathbf A\rightarrow \mathbf A_0$ then $\varphi \rightarrow \varphi _0$ hence $\mathbf M\rightarrow \mathbf M_0$ by the continuity of $F$ and Lemma \ref{converges}.
We have
\begin{align}
\frac{\mathbf A-\mathbf A_0}{\|\mathbf A-\mathbf A_0\|}\cdot \mathbf M_0\rightarrow 0\label{normal_proof}
\end{align}
as $\mathbf A\rightarrow \mathbf A_0$.
This says that $\mathbf M_0$ is a normal vector at $\mathbf A_0$ on $\mathcal A$, and the normal vector also varies continuously.
By (\ref{convex_proof}) and (\ref{normal_proof}), we conclude that the surface $\mathcal A$ is convex and smooth.
\end{proof}

The surface $\mathcal A$ is convex, but is not strictly convex if $k\geq 2$.
Let us consider $(A_1^2,\ldots ,A_{k-1}^2,x^2)\in \mathbf R_{\geq 0}^k$ such that $(A_1,\ldots ,A_{k-1})\not =\mathbf o$ is fixed for $k\geq 2$.
Then $A^*_k>0$ by Proposition \ref{star_range}, and if $0\leq x\leq A^*_k$, any $(A_1,\ldots ,A_{k-1},x)$ determines a common Jenkins-Strebel differential $\varphi $ by Lemma \ref{lower+}.
Therefore $(A_1^2/\|\varphi \|,\ldots ,A_{k-1}^2/\|\varphi \|,x^2/\|\varphi \|)\in \mathcal A$ for any $0\leq x\leq A^*_k$.
We conclude that there exist some Euclidean segments on $\mathcal A$ such that each one is parallel to an axis of $\mathbf R_{\geq 0}^k$.
This method is used again in \S \ref{dim2}.

\section{The surface of circumferences with Teichm\"uller theory}\label{with_Teich}

Let us consider the surface of circumferences on each Riemann surface with any complex structure.
More precisely, the surface is determined for any point of the Teichm\"uller space.
For the Teichm\"uller space, see \cite{ImaTan92}.

\subsection{Teichm\"uller space}

The \textit{Teichm\"uller space} of $\Sigma $ is represented by the set of all pairs of a Riemann surface $S$ and a quasiconformal mapping $f:\Sigma \rightarrow S$ with the equivalence relation.
Let $(S,f)$ denote such a pair.
Any two pairs $(S_1,f_1)$ and $(S_2,f_2)$ are equivalent if and only if there exists a conformal mapping $h:S_1\rightarrow S_2$ which is homotopic to $f_2\circ f_1^{-1}$.
We also use $(S,f)$ as the equivalence class, and denote by $\mathcal T$ the Teichm\"uller space.

The Teichm\"uller space $\mathcal T$ has the \textit{Teichm\"uller distance} $d_{\mathcal T}$ defined as follows:
for any $p_1=(S_1,f_1)$ and $p_2=(S_2,f_2)$ in $\mathcal T$,
\begin{align*}
d_{\mathcal T}(p_1,p_2)=\frac{1}{2}\log \inf _{h}K(h),
\end{align*}
where the infimum ranges over all quasiconformal mappings $h:S_1\rightarrow S_2$ which are homotopic to $f_2\circ f_1^{-1}$, and $K(h)$ is the maximal quasiconformal dilatation of each $h$.

We call a geodesic on $\mathcal T$ with respect to the Teichm\"uller distance $d_{\mathcal T}$ a \textit{Teichm\"uller geodesic}.
The Teichm\"uller space $\mathcal T$ is a unique geodesic space, that is, for any two points in $\mathcal T$, there is a unique Teichm\"uller geodesic which passes through the points.

\subsection{The definition of $\mathcal A$ on the Teichm\"uller space}

Let $S$ be a Riemann surface, $f:\Sigma \rightarrow S$ a quasiconformal mapping.
For any admissible curve family $\Gamma =\{\gamma _1,\ldots ,\gamma _k\}$ on $\Sigma $, the set $f(\Gamma )=\{f(\gamma _1),\ldots ,f(\gamma _k)\}$ is also an admissible curve family on $S$.
We denote by $JS(S,f(\Gamma ))$ the set of all Jenkins-Strebel differentials on $S$ of type $f(\Gamma )$ with the $0$-quadratic differential, and by $F_{S,f(\Gamma )}:\mathbf R_{\geq 0}^k\rightarrow JS(S,f(\Gamma ))$ the mapping defined similarly to $F$.
Let $\mathcal A(S,f(\Gamma ))$ denote the surface of circumferences for $f(\Gamma )$ on $S$.

The set $JS(\Gamma )$, the mapping $F$, and the surface $\mathcal A$ in \S \ref{surface_of_circumferences} are written by $JS(\Sigma ,\Gamma )$, $F_{\Sigma ,\Gamma }$, and $\mathcal A(\Sigma ,\Gamma )$ in the manner, respectively.

\begin{thm}\label{invariant}
The surface of circumferences is conformally invariant, that is, if $h:\Sigma \rightarrow S$ is a conformal mapping then $\mathcal A(\Sigma ,\Gamma )=\mathcal A(S,h(\Gamma ))$.
\end{thm}

\begin{proof}
Let $(A_1^2,\ldots ,A_k^2)\in \mathcal A(\Sigma ,\Gamma )$, and set $F_{\Sigma ,\Gamma }(A_1,\ldots ,A_k)=\varphi \in JS_0(\Sigma ,\Gamma )$.
For any conformal mapping $h:\Sigma \rightarrow S$, we consider the push-forward $h_*(\varphi )$, it is a Jenkins-Strebel differential on $S$ of type the admissible curve family $h(\Gamma )$, and satisfies $\|h_*(\varphi )\|=\|\varphi \|=1$, so that $h_*(\varphi )\in JS_0(S,h(\Gamma ))$.

Let $a_j=\inf _{\gamma '\sim \gamma _j}\int _{\gamma '}|\varphi |^{\frac{1}{2}}$ and $M_j$ be the modulus determined by $\varphi $, then $A_j\leq a_j$ for any $j=1,\ldots ,k$.
We notice that $a_j=\inf _{\gamma '\sim h(\gamma _j)}\int _{\gamma '}|h_*(\varphi )|^{\frac{1}{2}}$ and $M_j$ is also the modulus of the $j$-th characteristic ring domain of $h_*(\varphi )$ for any $j=1,\ldots ,k$.

We now consider $\varphi '=F_{S,h(\Gamma )}(A_1,\ldots ,A_k)\in JS(S,h(\Gamma ))$, and suppose that $a'_j=\inf _{\gamma '\sim h(\gamma _j)}\int _{\gamma '}|\varphi '|^{\frac{1}{2}}$ and $M'_j$ is the modulus determined by $\varphi '$, then also $A_j\leq a'_j$ for any $j=1,\ldots ,k$.
By Theorem \ref{Strebel_ext_prop},
\begin{align*}
\|\varphi '\|&=\sum _{j=1}^k{a'_j}^2M'_j\geq \sum _{j=1}^k{a'_j}^2M_j\geq \sum _{j=1}^kA_j^2M_j=\|\varphi \|\\
&=\|h_*(\varphi )\|=\sum _{j=1}^ka_j^2M_j\geq \sum _{j=1}^ka_j^2M'_j\geq \sum _{j=1}^kA_j^2M'_j=\|\varphi '\|,
\end{align*}
so it gives $\|\varphi '\|=\|h_*(\varphi )\|=1$ and $\varphi '=h_*(\varphi )$ then $F_{S,h(\Gamma )}(A_1,\ldots ,A_k)=h_*(\varphi )\in JS_0(S,h(\Gamma ))$.

We conclude that $(A_1^2,\ldots ,A_k^2)\in \mathcal A(S,h(\Gamma ))$.
Conversely, any point in $\mathcal A(S,h(\Gamma ))$ is also in $\mathcal A(\Sigma ,\Gamma )$ by using the similar method, then we have $\mathcal A(\Sigma ,\Gamma )=\mathcal A(S,h(\Gamma ))$.
\end{proof}

\begin{rem}
By the proof of the above theorem, we can also obtain $h_*(JS(\Sigma ,\Gamma ))=JS(S,h(\Gamma ))$ and $h_*\circ F_{\Sigma ,\Gamma }=F_{S,h(\Gamma )}$.
\end{rem}

We obtain the following corollary immediately by the above theorem and the definition of the Teichm\"uller space.

\begin{cor}\label{invariant_corollary}
The surface of circumferences is well-defined on the Teichm\"uller space, that is, if $(S_1,f_1)=(S_2,f_2)$ in $\mathcal T$ then $\mathcal A(S_1,f_1(\Gamma ))=\mathcal A(S_2,f_2(\Gamma ))$.
\end{cor}

Therefore, we can denote by $\mathcal A(p,\Gamma )$ as $\mathcal A(S,f(\Gamma ))$ for any $p=(S,f)\in \mathcal T$ and any admissible curve family $\Gamma $ on $\Sigma $, since it does not depend on any representation of $p$.
If $\Gamma $ is immediately, we denote by $\mathcal A(p)$ instead of $\mathcal A(p,\Gamma )$.

Finally, we put another property for the surface of circumferences.

\begin{prop}\label{any_vector}
Choose any positive numbers $(a_1,\ldots ,a_k)\in \mathbf R_{>0}^k$.
Then there exists a Riemann surface $S$ homeomorphic to $\Sigma $ such that there is a Jenkins-Strebel differential on $S$ of type $\Gamma $ whose characteristic ring domains have $a_1,\ldots ,a_k$ as their circumferences.
In other words, there exist a constant $\lambda >0$ and a point $p=(S,\id )\in \mathcal T$ such that $(\lambda a_1^2,\ldots ,\lambda a_k^2)\in \mathcal A(p)$ and the corresponding unit norm Jenkins-Strebel differential $F_{S,\Gamma }(\lambda ^{\frac{1}{2}}a_1,\ldots ,\lambda ^{\frac{1}{2}}a_k)$ has all non-degenerated characteristic ring domains.
\end{prop}

The following theorem is the key of the proof of Proposition \ref{any_vector}.

\begin{thm}{(\cite[Theorem 23.5]{Strebel84})}\label{punctures}
Let $p_1,\ldots ,p_m$ be (not necessarily all) punctures of $\Sigma $ and fix any $(a_1,\ldots ,a_m)\in \mathbf R_{>0}^m$.
Then there exists a unique holomorphic quadratic differential $\varphi $ on $\Sigma $ such that each component of $\Sigma -CT(\varphi )$ is a punctured disk with each puncture $p_j$, all vertical trajectories of $\varphi $ in each the punctured disk are closed and surround the puncture, such that their lengths are all $a_j$ with respect to the metric $|\varphi |^{\frac{1}{2}}$, for any $j=1,\ldots ,m$.
\end{thm}

The differential $\varphi $ constructed by the above theorem has poles of order $2$ in the punctures.
Therefore, the norm of $\varphi $ is infinity.

\begin{proof}[Proof of Proposition \ref{any_vector}]
Let $(a_1,\ldots ,a_k)\in \mathbf R_{>0}^k$ and consider the pinching each the curves $\Gamma $ on $\Sigma $ to nodes.
The resulting surface is a Riemann surface with nodes.
Let denote $X$ the surface minus all nodes, then $X$ is the union of punctured Riemann surfaces (possibly a single surface).
Let $p_j^1,p_j^2$ be the pair of punctures of $X$ corresponding to the collapsed $\gamma _j$ for any $j=1,\ldots ,k$.
We apply Theorem \ref{punctures} to $X$ such that it assigns the positive $a_j$ to $p_j^1,p_j^2$ for any $j=1,\ldots ,k$.
Then we obtain a holomorphic quadratic differential $\varphi '$ on $X$ such that each component of $X-CT(\varphi ')$ is a punctured disk with each puncture $p_j^l$, all vertical trajectories of $\varphi '$ in each the punctured disk are closed and surround the puncture, such that their lengths are all $a_j$ which does not depend on $l=1,2$, with respect to the metric $|\varphi '|^{\frac{1}{2}}$, for any $j=1,\ldots ,k$.

We next cut off each small punctured disk with puncture $p_j^l$ from $X$ for any $j=1,\ldots ,k$ and $l=1,2$ such that the boundary curve of the disk is a closed vertical trajectory of $\varphi '$.
The resulting surface has such the boundary curves, and each the pair of the curves have a common length $a_j$ with respect to the metric $|\varphi '|^{\frac{1}{2}}$, for any $j=1,\ldots ,k$.
Therefore, we can glue all the pairs of the boundaries, so obtain a standard Riemann surface $S$ with a Jenkins-Strebel differential $\varphi $ on $S$ of type $\Gamma $ induced by $\varphi '$, whose characteristic ring domains have the circumferences $a_j$ for any $j=1,\ldots ,k$, in particular, such ring domains are all non-degenerated.
We also obtain a natural mapping $\Sigma $ onto $S$ as the identity mapping, so that we regard $\varphi $ as in $JS(S,\Gamma )$.
It holds $F_{S,\Gamma }(a_1,\ldots ,a_k)=\varphi $ by Theorem \ref{Strebel_ext_prop}, then $(a_1^2/\|\varphi \|,\ldots ,a_k^2/\|\varphi \|)\in \mathcal A(S,\Gamma )$.
Set $p=(S,\id )$ and $\lambda =1/\|\varphi \|$, we obtain the desired result.
\end{proof}

\subsection{The surface of circumferences for two curves}\label{dim2}

The form of the surface of circumferences for an admissible curve family which only contains two curves, can be described.
Let $\Gamma =\{\gamma _1,\gamma _2\}$ be an admissible curve family on $\Sigma $, and we fix $\Gamma $ in this section.
Of coarse we exclude any $\Sigma $ such that $3g-3+n=1$.

We first treat the surface of circumferences for $\Gamma $ on $\Sigma $.
For simplicity, we again use the symbols $JS(\Gamma ), F, \mathcal A$ which are denoted in \S \ref{surface_of_circumferences}.

Let us consider the ray $\{(1,y^2)\mid y\geq 0\}$ in $\mathbf R_{\geq 0}^2$.
By Lemma \ref{lower+} and Proposition \ref{star_range}, there is $y^*>0$ such that every $F(1,y)$ is a same Jenkins-Strebel differential of type $\Gamma $, and it has the only one characteristic ring domain of type $\gamma _1$, for any $0\leq y\leq y^*$.
We denote the differential by $\varphi $.
The unit norm differential $\varphi _1=\varphi /\|\varphi \|$ is in $JS_0(\Gamma )$, and it satisfies $F(1/\|\varphi \|^{\frac{1}{2}},y/\|\varphi \|^{\frac{1}{2}})=\varphi _1$ for any $0\leq y\leq y^*$.
This implies that the segment $\{(1/\|\varphi \|,y^2/\|\varphi \|)\mid 0\leq y\leq y^*\}$ is contained in $\mathcal A$.

We would like to give explicit descriptions of the coordinates $1/\|\varphi \|$ and ${y^*}^2/\|\varphi \|$ which are appeared in boundaries of the segment.

The differential $\varphi _1$ also has the only one characteristic ring domain of type $\gamma _1$.
Then let $a_1,a_2,M_1>0,M_2=0$ denote the quantities by $\varphi _1$, so that $a_1=1/\|\varphi \|^{\frac{1}{2}}$ and $a_1^2M_1+a_2^2\cdot 0=\|\varphi _1\|=1$, therefore $1/\|\varphi \|=a_1^2=1/M_1=\ext _{\Sigma }(\gamma _1)$, this is the extremal length of $\gamma _1$ on $\Sigma $.

On the other hand, by Proposition \ref{characterize}, $y^*=\inf _{\gamma '\sim \gamma _2}\int _{\gamma '}|\varphi |^{\frac{1}{2}}$, so that ${y^*}^2/\|\varphi \|=\{\inf _{\gamma '\sim \gamma _2}\int _{\gamma '}|\varphi _1|^{\frac{1}{2}}\}^2$.
Let $H(\varphi _1)$ be the horizontal measured foliation on $\Sigma $ induced by $\varphi _1$.
It is foliated by vertical trajectories of $-\varphi _1$ with the transverse measure $|dy|$, where $z=x+iy$ is the $\varphi _1$-coordinate.
The quantity $\inf _{\gamma '\sim \gamma _2}\int _{\gamma '}|\varphi _1|^{\frac{1}{2}}$ is equal to the infimum of the lengths of all simple closed curves homotopic to $\gamma _2$ measured by $|dy|$, see Remark \ref{converges_problem}.
That is,
\begin{align*}
\inf _{\gamma '\sim \gamma _2}\int _{\gamma '}|\varphi _1|^{\frac{1}{2}}=i(H(\varphi _1),\gamma _2),
\end{align*}
where $i(\cdot ,\cdot )$ is the intersection number function.
This implies that ${y^*}^2/\|\varphi \|=i(H(\varphi _1),\gamma _2)^2$.

\begin{rem}
Any unit norm Jenkins-Strebel differential in $JS_0(\{\gamma _1\})$ is $\varphi _1$, that is,  $JS_0(\{\gamma _1\})=\{\varphi _1\}$, so that $\varphi _1$ depends only on the underlying Riemann surface $\Sigma $ and the simple closed curve $\gamma _1$.
This is also induced by Theorem \ref{Strebel_ext_prop}.
\end{rem}

Similarly, we use $\{(x^2,1)\mid x\geq 0\}$ and the same methods, so can find a suitable segment in $\mathcal A$, described by $\varphi _2\in JS_0(\{\gamma _2\})$.
In summary we obtain the segments
\begin{align*}
&L_1=\{(\ext _{\Sigma }(\gamma _1),y^2)\mid 0\leq y\leq i(H(\varphi _1),\gamma _2)\},\\
&L_2=\{(x^2,\ext _{\Sigma }(\gamma _2))\mid 0\leq x\leq i(H(\varphi _2),\gamma _1)\}
\end{align*}
which are contained in $\mathcal A$.

Consequently, the set $\mathcal A$ is a convex curve in $\mathbf R_{\geq 0}^2$ which has the segments $L_1,L_2$ orthogonal to each axis of $\mathbf R_{\geq 0}^2$ by combining Theorem \ref{surface_A} and the above argument.
We recall that any vector of moduli $(M_1,M_2)$ induced by the Jenkins-Strebel differential corresponding to a point in $\mathcal A$ is a normal vector of $\mathcal A$ by the proof of Theorem \ref{surface_A}.
This means that any Jenkins-Strebel differential corresponding to a point in $\mathcal A-(L_1\cup L_2)$ has exactly two non-degenerated characteristic ring domains.
Furthermore, by considering the equality condition of the inequality (\ref{convex_proof}) in the proof of Theorem \ref{surface_A}, then we can know that the sub curve $\mathcal A-(L_1\cup L_2)$ is strictly convex.
Figure \ref{A} represents the form of $\mathcal A$.

	\begin{figure}[!ht]
	\centering
	\includegraphics[keepaspectratio, scale=0.85]
	{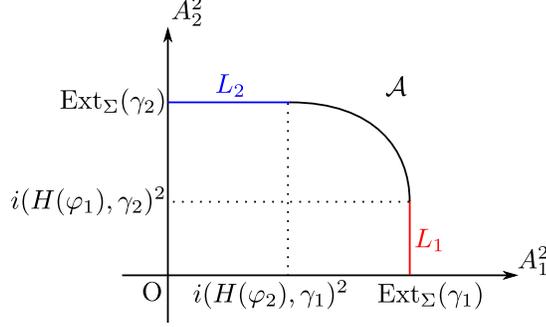}
	\caption{The surface of circumferences $\mathcal A$ for two curves $\{\gamma _1,\gamma _2\}$.}
	\label{A}
	\end{figure}

\begin{rem}\label{range_of_A}
We notice that $\mathcal A-(L_1\cup L_2)$ is not empty.
Indeed, the height problem or the moduli problem guarantees the existence of a Jenkins-Strebel differential with all non-degenerated characteristic ring domains.
Also, $\mathcal A-(L_1\cup L_2)\not =\mathcal A$ by construction.
\end{rem}

We obtain the following theorem.

\begin{thm}\label{form_of_A}
The surface of circumferences $\mathcal A$ for any admissible curve family $\{\gamma _1,\gamma _2\}$ consists of two curves forms a convex curve in $\mathbf R_{\geq 0}^2$.
It has two segments $L_1$ and $L_2$ orthogonal to each axis of $\mathbf R_{\geq 0}^2$ such that any Jenkins-Strebel differential $F(\lambda A_1,\lambda A_2)$ on $\Sigma $ for $(A_1^2,A_2^2)\in \mathcal A-(L_1\cup L_2)$ and $\lambda >0$ has exactly two non-degenerated characteristic ring domains of type $\{\gamma _1,\gamma _2\}$.
\end{thm}

\begin{cor}\label{extremal_lengths}
For any admissible curve family $\{\gamma _1,\gamma _2\}$, there exists the Jenkins-Strebel differential on $\Sigma $ of type $\{\gamma _1,\gamma _2\}$ whose characteristic ring domains are all non-degenerated and have $\ext _{\Sigma }^{\frac{1}{2}}(\gamma _1),\ext _{\Sigma }^{\frac{1}{2}}(\gamma _2)$ as the circumferences.
\end{cor}

Let any $p=(S,f)\in \mathcal T$ be given.
For any holomorphic quadratic differential $\varphi $ on $S$, we denote by $H(\varphi )$ the horizontal measured foliation on $S$ induced by $\varphi $.
Let $\mathcal A(p)$ be the surface of circumferences for $f(\Gamma )$ on $S$ defined in the previous section, so the segments $L_1(p)$ and $L_2(p)$ are defined similarly, that is,
\begin{align*}
&L_1(p)=\{(\ext _S(f(\gamma _1)),y^2)\mid 0\leq y\leq i(H(\varphi _1(p)),f(\gamma _2))\},\\
&L_2(p)=\{(x^2,\ext _S(f(\gamma _2)))\mid 0\leq x\leq i(H(\varphi _2(p)),f(\gamma _1))\},
\end{align*}
where $\varphi _j(p)$ is the unique element of $JS_0(S,\{f(\gamma _j)\})$ for any $j=1,2$.
Then $\mathcal A(p)$ is a convex curve and contains $L_1(p)$ and $L_2(p)$.
The similar results for Theorem \ref{form_of_A} and Corollary \ref{extremal_lengths} are also valid in the case of any $p\in \mathcal T$, see Figure \ref{A(p)}.

We now consider how large the range of $\mathcal A(p)-(L_1(p)\cup L_2(p))$ when $p$ ranges over all points in the Teichm\"uller space $\mathcal T$.
As in Figure \ref{A(p)}, we define the following angle
\begin{align*}
\Theta (p)=\arctan \frac{\ext _S(f(\gamma _2))}{i(H(\varphi _2(p)),f(\gamma _1))^2}-\arctan \frac{i(H(\varphi _1(p)),f(\gamma _2))^2}{\ext _S(f(\gamma _1))}.
\end{align*}
By Remark \ref{range_of_A} for any $p\in \mathcal T$, $\Theta (p)\not =0,\pi /2$.
Then this angle determines a mapping $\Theta :\mathcal T\rightarrow (0,\pi /2)$.

	\begin{figure}[!ht]
	\centering
	\includegraphics[keepaspectratio, scale=0.85]
	{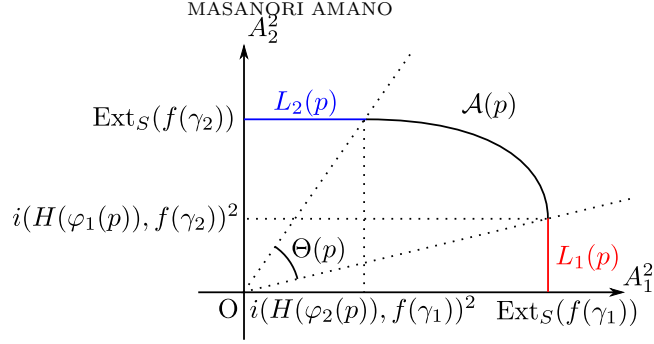}
	\caption{The surface of circumferences $\mathcal A(p)$ and the angle $\Theta (p)$.}
	\label{A(p)}
	\end{figure}

The angle $\Theta (p)$ is not bounded away from $0$.

\begin{thm}\label{sequence0}
We have $\inf _{p\in \mathcal T}\Theta (p)=0$.
\end{thm}

\begin{proof}
We recall that $\Gamma =\{\gamma _1,\gamma _2\}$ is an admissible curve family on $\Sigma $.
Let $\delta $ be a non-zero and non-peripheral simple closed curve on $\Sigma $ such that $i(\delta ,\gamma _1)>0$ and $\delta $ does not intersect with $\gamma _2$.
Let $\tau $ be a Dehn twist about $\delta $.
We notice that $\tau (\gamma _2)=\gamma _2$ as the homotopy class.

For any $n\in \mathbf N$, let $p_n=(\Sigma ,\tau ^n)\in \mathcal T$, then $\varphi _1(p_n)\in JS_0(\Sigma ,\{\tau ^n(\gamma _1)\}),\ \varphi _2(p_n)\in JS_0(\Sigma ,\{\gamma _2\})$, so that
\begin{align*}
&\frac{i(H(\varphi _1(p_n)),\tau ^n(\gamma _2))^2}{\ext _{\Sigma }(\tau ^n(\gamma _1))}
=\frac{i(H(\varphi _1(p_n)),\gamma _2)^2}{\ext _{\Sigma }(\tau ^n(\gamma _1))}
\leq \frac{\ext _{\Sigma }(\gamma _2)}{\ext _{\Sigma }(\tau ^n(\gamma _1))}\rightarrow 0,\\
&\frac{\ext _{\Sigma }(\tau ^n(\gamma _2))}{i(H(\varphi _2(p_n)),\tau ^n(\gamma _1))^2}
=\frac{\ext _{\Sigma }(\gamma _2)}{i(H(\varphi _2(p_n)),\tau ^n(\gamma _1))^2}\rightarrow 0
\end{align*}
as $n\rightarrow \infty $, since the Riemann surface $\Sigma $, the curve $\gamma _2$, and so $\varphi _2(p_n)$ do not depend on $n$, and the inequality $i(H(\varphi _1(p_n)),\gamma _2)^2\leq \ext _{\Sigma }(\gamma _2)$ comes from the form of $\mathcal A(p_n)$, see Figure \ref{A(p)}, or is implied by Minsky's inequality, see \cite{Minsky96}.
Therefore, $\Theta (p_n)\rightarrow 0$ as $n\rightarrow \infty $.
\end{proof}

Conversely, we do not know that $\sup _{p\in \mathcal T}\Theta (p)=\pi /2$ is true or false.

The mapping $\Theta $ takes any angle in the interval $(0,\sup _{p\in \mathcal T}\Theta (p))$.

\begin{prop}\label{dense}
We have $\Theta (\mathcal T)\supset (0,\sup _{p\in \mathcal T}\Theta (p))$.
\end{prop}

It comes from the following lemma.

\begin{lemma}\label{continuous}
The mapping $\Theta :\mathcal T\rightarrow (0,\pi /2)$ is continuous.
\end{lemma}

To prove this lemma, we use the two homeomorphisms from the work of Hubbard and Masur \cite{HubMas79}.
Let $Q$ be a vector bundle over $\mathcal T$ of all holomorphic quadratic differentials.
The fiber above $p=(S,f)\in \mathcal T$ is the set of all holomorphic quadratic differentials on $S$.
Let $\mathcal {MF}$ be the set of all (Whitehead equivalence classes of) measured foliations on $\Sigma $.

\begin{thm}[\cite{HubMas79}]\label{HandM1}
The mapping $Q\rightarrow \mathcal T\times \mathcal {MF}$ which sends any $\varphi \in Q$ to the pair of its corresponding point $p=(S,f)\in \mathcal T$ and the pulled back horizontal measured foliation $f^*(H(\varphi ))$ on $\Sigma $ determined by $\varphi $, is a homeomorphism.
\end{thm}

\begin{thm}[\cite{HubMas79}]\label{HandM2}
Let $\Gamma =\{\gamma _1,\ldots ,\gamma _k\}$ be an admissible curve family on $\Sigma $.
Let $JS^{ess}(\Gamma )\subset Q$ be the space of all Jenkins-Strebel differentials on $S$ with all non-degenerated characteristic ring domains of type $f(\Gamma )$ which are above any $p=(S,f)\in \mathcal T$.
The mapping $JS^{ess}(\Gamma )\rightarrow \mathcal T\times \mathbf R_{>0}^k$ which sends any $\varphi \in JS^{ess}(\Gamma )$ to the pair of its corresponding point $p$ and the heights $(b_1,\ldots ,b_k)$ of the characteristic ring domains determined by $\varphi $, is a homeomorphism.
\end{thm}

\begin{proof}[Proof of Lemma \ref{continuous}]
It is well-known that each extremal length function is continuous on $\mathcal T$.
Then this is enough to prove that the mapping $p\mapsto i(H(\varphi _i(p)),f(\gamma _j))$ is continuous for $(i,j)=(1,2)$ or $(2,1)$ for any $p=(S,f)\in \mathcal T$.

We apply the above two theorems to $\varphi _i(p)$.
By Theorem \ref{HandM1}, the mapping $\varphi _i(p)\mapsto (p,f^*(H(\varphi _i(p))))$ is continuous.
Next, since the height of the only one characteristic ring domain of $\varphi _i(p)$ is $1/\ext _S^{\frac{1}{2}}(f(\gamma _i))$, the mapping $(p,1/\ext _S^{\frac{1}{2}}(f(\gamma _i)))\mapsto \varphi _i(p)$ is continuous by Theorem \ref{HandM2}.
Therefore, the composition
\begin{align*}
p=(S,f)\mapsto \left(p,\frac{1}{\ext _S^{\frac{1}{2}}(f(\gamma _i))}\right)\mapsto \varphi _i(p) \mapsto (p,f^*(H(\varphi _i(p))))\mapsto f^*(H(\varphi _i(p)))
\end{align*}
is continuous.
The intersection number function $i:\mathcal {MF}\times \mathcal {MF}\rightarrow \mathbf R_{\geq 0}$ is also continuous, so we have that the mapping
\begin{align*}
p\mapsto i(f^*(H(\varphi _i(p))),\gamma _j)=i(H(\varphi _i(p)),f(\gamma _j))
\end{align*}
is also.
\end{proof}

\begin{proof}[Proof of Proposition \ref{dense}]
For any $x\in (0,\sup _{p\in \mathcal T}\Theta (p))$, there are two points $p,q$ in $\mathcal T$ such that $\Theta (p)<x<\Theta (q)$ by Theorem \ref{sequence0}.
Let $r:[0,t]\rightarrow \mathcal T$ be a Teichm\"uller geodesic segment on $\mathcal T$ which joins $p=r(0)$ and $q=r(t)$.
By Lemma \ref{continuous}, the composition $\Theta \circ r:[0,t]\rightarrow (0,\pi /2)$ is continuous.
Then there is $c\in (0,t)$ such that $x=\Theta \circ r(c)$.
\end{proof}

\bibliographystyle{plain}
\bibliography{references}
\end{document}